\documentclass[journal]{IEEEtran}
\usepackage{amsmath,epsfig}
\usepackage{etoolbox,enumitem}
\usepackage{amsmath,graphicx,subcaption}
\usepackage{bm,amsmath,amssymb,amsfonts,epsfig,amsthm,color,algpseudocode}
\usepackage{cite,epstopdf}
\usepackage{enumerate,url,algpseudocode}
\usepackage[]{algorithm2e}
\usepackage{soul}
\usepackage{hyperref}
\usepackage{amsthm}
\usepackage{multirow}
\usepackage{mathtools}
\usepackage{dblfloatfix}
\usepackage{booktabs,array}
\usepackage[font=small]{caption}

\newcounter{define}
\newcounter{assume}
\newcounter{exampl}

\newcommand{\ahmed}[1]{\textcolor{black}{#1}}

\newtheorem{theorem}{Theorem}

\newtheorem{definition}[define]{Definition}
\newtheorem{assumption}[assume]{Assumption}
\newtheorem{example}[exampl]{Example}

\newtheorem*{remark*}{Remark}

\begin{document}
	
	\title{Physics-Aware Neural Networks for Distribution System State Estimation}
	\author{
		Ahmed S. Zamzam, \IEEEmembership{Student Member,~IEEE,} and
		Nicholas D. Sidiropoulos,~\IEEEmembership{Fellow,~IEEE}
		\thanks{
			A. S. Zamzam is with the ECE Dept., Univ. of Minnesota, Minneapolis, MN 55455, USA. N. D. Sidiropoulos is with the Department of Electrical and Computer Engineering, University of Virginia, Charlottesville, VA 22904.
		}
	}
	\maketitle
	
\begin{abstract}
The distribution system state estimation problem seeks to determine the network state from available measurements. Widely used Gauss-Newton approaches are very sensitive to the initialization and often not suitable for real-time estimation. Learning approaches are very promising for real-time estimation, as they shift the computational burden to an offline training stage. Prior machine learning approaches to power system state estimation have been electrical model-agnostic, in that they did not exploit the topology and physical laws governing the power grid to design the architecture of the learning model. In this paper, we propose a novel learning model that utilizes the structure of the power grid. The proposed neural network architecture reduces the number of coefficients needed to parameterize the mapping from the measurements to the network state by exploiting the separability of the estimation problem. This prevents overfitting and reduces the complexity of the training stage. We also propose a greedy algorithm for phasor measuring units placement that aims at minimizing the complexity of the neural network required for realizing the state estimation mapping. Simulation results show superior performance of the proposed method over the Gauss-Newton approach.
\end{abstract}


\section{Introduction}
Distribution system state estimation (DSSE) is an important task for monitoring and control of distribution networks. DSSE takes as input a set of measurements of physical quantities in the network and provides an estimate of the system state, i.e., nodal voltages. Due to the rapid introduction of volatile renewable energy sources and controllable loads, modern distribution grids are challenged by unusual fluctuations in the operational conditions. Therefore, accurate real-time monitoring of distribution networks becomes increasingly pivotal in order to ensure reliable and optimal operation of the grid.

The DSSE task is often performed by formulating a weighted least squares (WLS) optimization problem~\cite{baran1994, li1996state, singh2009choice, kekatos2013distributed,  Wang2018}. In~\cite{baran1994, li1996state}, a WLS-based DSSE solver was proposed using a three-phase nodal voltage formulation. Recently, a Wirtinger calculus-based approach was devised in~\cite{dzafic2018hybrid} to solve the problem in the complex domain to reduce computational complexity. Another DSSE solver based on the branch current formulation was proposed in~\cite{baran1995branch, wang2004revised} which reduces the computational complexity of the problem when the system features only solidly-grounded wye-connected loads. All these solvers rely solely on physics-based models which usually lead to nonconvex optimization problems that are computationally expensive. 

Exploiting valuable information from abundant real-time and historical data, data-driven approaches hold the promise to significantly enhance monitoring accuracy and improve the performance of distribution networks. To that end, neural network approaches have been used to estimate the bus injections from the real-time measurements in~\cite{manitsas2012distribution}. The estimated bus injections can be used as pseudo-measurements to compensate for the scarcity of real-time measurements. In addition, plain feed-forward neural networks (NN) were proposed to estimate the network state from the measurements in~\cite{Barbeiro2014}. This approach reduces the complexity of the state estimation task to matrix-vector multiplications by shifting the computational burden to an off-line training stage utilizing historical or simulated data. It is often challenging to avoid exploding or vanishing gradients while training these feed-forward NNs, and thus the provided estimates are less accurate than any optimization-based approach. 
A joint optimization/learning approach was proposed in~\cite{Zamzam-2018}. Since GN works very well when given a proper initialization, the key is to {\em learn to initialize} a Gauss-Newton solver. This entails a special design of the learning cost function, but in turn a shallow NN suffices to learn to initialize, keeping sample complexity and run-time complexity low, while benefiting from the high accuracy of properly initialized GN. Different from~\cite{Zamzam-2018} the authors of~\cite{zhang2018real} devised a learning approach where a deep NN is constructed by unfolding an iterative solver for the
least-absolute-value formulation of the state estimation problem in transmission networks~\cite{wang2017robust}. 

All past learning models for state estimation overlook the physics of the underlying distribution network, hence leading to over-parameterization of the mapping from the measurements to the network states.
In order to utilize our knowledge of the physical system that governs the relationship between the network states and the measurements, we propose a novel neural network architecture exploiting the distribution network structure. We start by showing that owing to disparity in the accuracy of the measurements, i.e., the $ \mu $-phasor-measuring unit ($ \mu $PMU) measurements are far more accurate than any other measurements in the network, the DSSE problem can be (approximately) partitioned into smaller problems. Therefore, the estimation of the state (voltage) at a certain bus in the network can be done using the measurements taken at the partition/partitions where this bus is located. In this paper, we formally describe the partitioning of the DSSE problem that results from installing $ \mu $PMUs in the distribution network. In addition, the quality of the partitioning resulting from a certain installation of $ \mu $PMUs in the network can be assessed using the diameter measure which is related to the size of partitions generated. We propose a greedy algorithm for installing $ \mu $PMUs in the distribution network to minimize the diameter of the resulting partitioning. Simulations results on the IEEE-37 distribution feeder show that the proposed learning approach achieves superior performance in terms of estimation accuracy. Also, the running time is in the order of milliseconds which allows for real-time monitoring of the distribution network.

The main idea of the proposed NN architecture is to zero out the weights of the measurements taken outside a particular partition when calculating the estimate of the network states at the buses inside that partition. To do so, the structure of the network admittance matrix is embedded on the weights matrix that maps the NN iterates. 
{The uderlying physical model that governs the operation of the distribution network is utilized to sparsify the learning model,} and thus the proposed model is called \emph{Physics-aware neural network} (PAWNN) where the pruning is done in a deterministic manner before training. In the proposed architecture, the output of a $ K $-layer PAWNN that relates to the estimated voltage at a certain bus is only a function of the measurements taken at most $ K $ hops away from this bus. Therefore, in order to realize the mapping from the measurements to the states, the number of layers in the PAWNN has to match the diameter of the partitioning resulting from the $ \mu $PMUs placement. In our simulations, we show that the proposed greedy algorithm for $ \mu $PMUs placement achieves near optimal performance in terms of minimizing the diameter of the resulting partitioning.

{The proposed NN architecture reduces the number of the trainable parameters, which prevents over-fitting. Also, it inherently provides robustness as any reconfiguration in the network topology will only affect the state estimates in this specific directly affected areas. For example, if the neural network has $K$ layers, any measurement contributes only to the state estimates at buses that are at most $ K $ hops away from the location where the measurement is taken. Similarly, for the case of measurement outliers resulting from damaged meters or communication failures, the estimation of the state at distant locations from the outliers measurements will not be affected.}

A different approach was introduced in~\cite{gao2018large}. A learnable graph convolutional NNs model was proposed where an automatic selection of a fixed number of neighboring nodes for each feature based on value ranking in order to transform graph data into grid-like structures in 1-D format, thereby enabling the use of regular convolutional neural networks. However, this approach can not be directly applied to the state estimation problem as the distribution of the measurements is usually not uniform across the nodes.

The rest of this paper is structured as follows. Section II outlines the DSSE formulation. Section III shows the partitioning of the DSSE problem based on the $ \mu $PMU placement. Section IV introduces our novel PAWNN architecture, and presents our algorithm for $ \mu $PMU placement. Simulated tests are presented in Section V, and the paper is concluded in Section VI. 

\noindent{\bf Notation}: sets are denoted by calligraphic letters, and matrices (vectors) are denoted by boldface capital (small) letters; $(\cdot)^T$, $ \overline{(\cdot)} $ and $(\cdot)^H$ stand for transpose, complex-conjugate and complex-conjugate transpose, respectively; and $|(\cdot)|$ denotes the magnitude of a number or the cardinality of a set.

\section{Distribution System State Estimation Problem}
We consider a multi-phase distribution feeder consisting of $ N $ buses and $ L $ lines that can be modeled as a graph $ \mathcal{G} := (\mathcal{N}, \mathcal{L}) $, where $ \mathcal{N}:=\{1, 2, \ldots, N\} $ comprises all the buses, and $ \mathcal{L} \subseteq \mathcal{N} \times \mathcal{N} $ represents the lines in the network. Let $ {\bf v}_n = [v_{n,a}, v_{n,b}, v_{n,c}]^T$ represent the voltage at all the phases of bus $ n $. Then, define $ {\bf v} := [{\bf v}_1^T, {\bf v}_2^T 	, \ldots, {\bf v}_n^T]^T $ which collects the voltages at all the buses $ n\in\mathcal{N} $. For each line $ (l, m)\in \mathcal{L} $, let $ {\bf Z}_{lm} = {\bf Y}_{lm}^{-1} $ denote the phase impedance matrix, and let $ \overline{\bf Y}_{lm} $ denote the shunt admittance matrix in the $ \pi $-equivalent model.

The DSSE problem aims to recover the system state vector $ {\bf v} \in \mathbb{C}^{3N}$ from real-time measurements, and pseudo-measurements. Due to the scarcity of real-time measured quantities, pseudo-measurements which relate to forecasted loads and renewable generation are used as surrogates. Naturally, the measurement noise level of the pseudo-measurements is higher than the noise level corresponding to the real-time measured quantities.

Advanced metering infrastructure, supervisory control and data acquisition (SCADA), and $ \mu $PMUs that are placed at some locations in the distribution network provide real-time measurements. The measured quantities are modeled as
\begin{equation}\label{eq:meas1}
\tilde{z}_\ell = \tilde{h}_\ell({\bf v}) + w_\ell, \qquad 1 \leq \ell \leq L_m
\end{equation}
where $ w_\ell $ accounts for the measurement noise and the modeling inaccuracies. We assume that $ w_\ell $ is a zero-mean Gaussian noise with known variance of $ \tilde{\sigma}_\ell^2 $.
The functions $ \tilde{h}_\ell({\bf v}) $ are the measurement synthesizing functions, and can be either linear or quadratic relationships. Later, we will discuss the specific form of $ \tilde{h}_\ell({\bf v}) $. 
Additionally, pseudo-measurements are obtained by utilizing load and renewable generation forecasting methods which can help in enhancing the observability of the system states. The forecasted quantities are usually noisy and adhere to 
\begin{equation}\label{eq:meas2}
\check{z}_\ell = \check{h}_\ell ({\bf v}) + u_\ell, \qquad 1 \leq \ell \leq L_s
\end{equation}
where $ u_\ell $ amounts for the  zero-mean forecast error which is assumed to have known variance $ \check{\sigma}_\ell^2 $. The forecasted quantities $ \check{z}_\ell ${'s} are power-related, and hence, they can be modeled as quadratic functions of the state variable $ {\bf v} $. The value of the measurement noise variance {$ \tilde{\sigma}_\ell^2 $} 
is determined by the accuracy of the measuring equipment and modeling errors, while the variance of the forecast error can be obtained through historical forecast data.

Define $ {\bf z}$ to be a vector comprising all the real-time measurements and pseudo-measurements, and let $ {\bf h}({\bf v}) : \mathbb{C}^{3N}\rightarrow \mathbb{R}^{L_m+L_s} $ be the mapping from the voltage (state) vector $ {\bf v} $ to the measurements.
The weighted least-squares formulation of the DSSE problem can be cast as follows
\begin{align} \label{eq:psse}\nonumber
\min_{\bf v}\ J({\bf v}) &= \sum_{\ell = 1}^{L_m} \tilde{w}_\ell \big(\tilde{z}_\ell - \tilde{h}_\ell ({\bf v})\big)^2 + \sum_{\ell = 1}^{L_s} \check{w}_\ell \big(\check{z}_\ell - \check{h}_\ell ({\bf v})\big)^2\\
&= \ ({\bf z} - {\bf h}({\bf v}))^T {\bf W} ({\bf z} - {\bf h}({\bf v}))
\end{align}
where the values of $ \tilde{w}_\ell $ and $ \check{w}_\ell $ are inversely proportional to $ \sigma_\ell^2 $ and $ \check{\sigma}_\ell^2 $, respectively. The optimization problem~\eqref{eq:psse} is {non-convex} due to the nonlinearity of the measurement mappings $ {\bf h}({\bf v}) $ inside the squares. 

	
The measurement functions $ \check{h}({\bf v}) $ and $ \tilde{h}({\bf v}) $ consist of:
	
\begin{itemize}
	\item {\it phasor measurements} which comprise the complex nodal voltages $ {\bf v}_n $, and/or current flows $ {\bf i}_{lm} $ and are obtained using the $ \mu $PMUs. Therefore, the corresponding measurement synthesizing function is linear in the state variable $ {\bf v} $. Each complex measurement is represented as two real measurements, i.e., the real and imaginary parts of the measured complex quantity. For example, the real and imaginary part of the complex nodal voltage at bus $ n $ for phase $ \phi $ are given by the following measurement synthesizing functions
	\begin{equation}
	\Re\{v_{n,\phi}\} =\ \frac{1}{2}\ {\bf e}_{\phi}^T\ ({\bf v}_n + \overline{\bf v}_n),
	\end{equation}
	\begin{equation}
	\Im\{v_{n,\phi}\} =\ \frac{1}{2j}\ {\bf e}_{\phi}^T\ ({\bf v}_n -  \overline{\bf v}_n)
	\end{equation}
	where $ {\bf e}_{\phi} $ is the $ \phi $-th canonical basis in $ \mathbb{R}^{3} $. Similarly, the real and imaginary parts of the complex current flow measurements can be written as
	\begin{equation}
	\Re\{i_{lm,\phi}\} = \frac{1}{2}\ {\bf e}_{\phi}^T\ \big( {\bf Y}_{lm} ({\bf v}_l - {\bf v}_m) +  \overline{\bf Y}_{lm} (\overline{\bf v}_l  - \overline{\bf v}_m ) \big),
	\end{equation}
	\begin{equation}
	\Im\{i_{lm,\phi}\} = \frac{1}{2j}\ {\bf e}_{\phi}^T\ \big( {\bf Y}_{lm} ({\bf v}_l - {\bf v}_m) -  \overline{\bf Y}_{lm}  (\overline{\bf v}_l  - \overline{\bf v}_m ) \big).
	\end{equation}
	\item {\it real-valued measurements} which include voltage magnitudes $ |{v}_{n,\phi}| $, current magnitudes $ |i_{lm,\phi}| $, and real and reactive power flow and injection measurements $ p_{lm,\phi}, q_{lm,\phi}, p_{n, \phi}, q_{n, \phi} $. These measurements are usually acquired by advanced metering infrastructure (AMI), SCADA systems, or $ \mu $PMUs. The real-valued measurements are nonlinearly related to the state variable $ {\bf v} $. The measured voltage magnitude square, the current magnitude square, and active and reactive power flows can be represented as quadratic functions of the state variable $ {\bf v} $, see~\cite{DallAnese13, Zamzam-2018}. Hence, all the real-valued measurements can be written as quadratic measurements of the state variable $ {\bf v} $.
	\item {\it pseudo-measurements} can obtained through load and renewable energy generation forecast methods which aim to estimate these quantities exploiting historical data and locational information. These measurements are often less accurate than real-time measurements, and hence, low weights are assigned to their corresponding terms in the WLS formulation. Like the real-valued measurements, the functions governing the mapping from the state variable to the forecasted load and renewable energy source injections can be formulated as quadratic functions~\cite{DallAnese13, Zamzam-2016}. 
\end{itemize}

	
Therefore, any measurement synthesizing function $ h_\ell({\bf v}) $ can be written in the following form
\begin{equation}
	h_\ell ({\bf v}) = {\overline{\bf v} }^T {\bf D}_\ell {\bf v} + {\bf c}_\ell^T {\bf v} + {\overline{\bf c}_\ell }^T \overline{\bf v} 
\end{equation}
where $ {\bf D}_\ell $ is a Hermitian matrix. This renders $ J({\bf v}) $ a fourth order function of the state variable, {which is} very challenging to optimize.


\section{Partitioned DSSE}
This section presents the required background for the partitioning of the DSSE problem that results from installing $ \mu $PMUs that provide very accurate measurements. First, we introduce the vertex-cut partitioning which divides the edges of the graph into disjoint sets. Then, we show that installing $ \mu $PMUs in the network results in a vertex-cut partitioning of the DSSE problem. Throughout this section, we use vertex, node, and bus interchangeably. Similarly, we use edge and line to refer to any connection in the graph.
\label{sec:PDSSE}
\begin{definition}
	An \emph{articulation vertex} of a connected graph is a vertex whose removal disconnects the graph~\cite[\S 2.4]{Chartrand-1985}.
\end{definition}

According to the definition, all the vertices in a tree graph are articulation points since removing any vertex disconnects the graph. Next, we define vertex-cut partitioning which partitions the set of the edges in the graph into multiple disjoint subsets.
\begin{definition}
	A {\em vertex-cut partitioning} refers to a partitioning of the edge set $ \mathcal{L} $ into $ K $ subsets $ \mathcal{L}_k $, such that $ \mathcal{L}_k \in \mathcal{L} $, $ \cup_{1\leq k\leq K} \  \mathcal{L}_k = \mathcal{L}$, and $ \mathcal{L}_k \cap \mathcal{L}_{k'} = \phi $ for $ k \neq k' $. Any vertex that holds an endpoint of an edge $ (l, m) \in \mathcal{L}_k $ is in $ \mathcal{N}_k $.
\end{definition}
Examples of vertex-cut partitioning are depicted in Fig.~\ref{fig:v-c}. If the original graph is a tree, then the number of disjoint subsets of edges that result from choosing a vertex to be cut is equal to the number of edges connected to that vertex. In Fig.~\ref{fig:v-c}(a) and (b), we show the two graphs resulting from cutting a vertex that has two edges in the original graph. When the vertex to be cut has three edges as in Fig.~\ref{fig:v-c} (c), the number of the resulting subgraphs is three where the cut-vertex is replicated in all the subgraphs. 

\begin{figure*}[t!]
	\centering
	\begin{subfigure}[b]{0.23\textwidth}\
		\centering
		\includegraphics[height=1.2in]{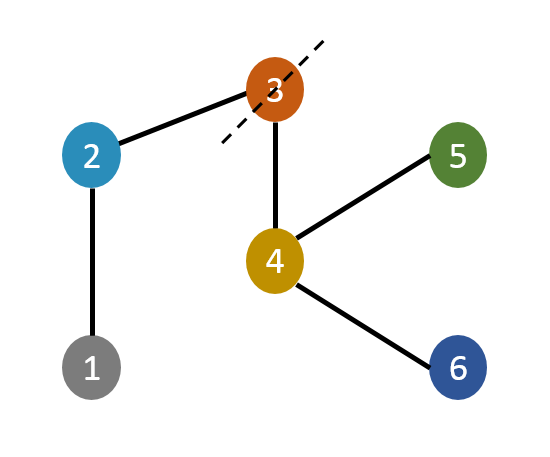}
		\caption{}
		\label{fig:g1}
	\end{subfigure}%
	~ 
	\begin{subfigure}[b]{0.23\textwidth}
		\centering
		\includegraphics[height=1.2in]{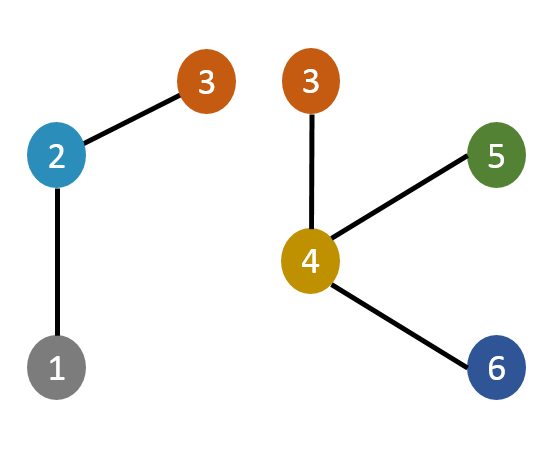}
		\caption{}
		\label{fig:v-c1}
	\end{subfigure}
	~
	\begin{subfigure}[b]{0.23\textwidth}
		\centering
		\includegraphics[height=1.2in]{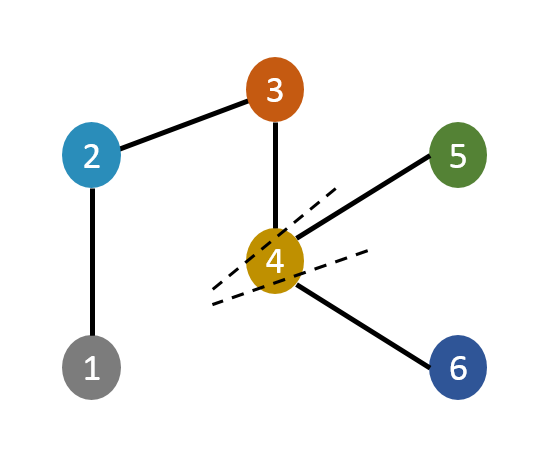}
		\caption{}
				\label{fig:g2}
	\end{subfigure}%
	~ 
	\begin{subfigure}[b]{0.23\textwidth}
		\centering
		\includegraphics[height=1.2in]{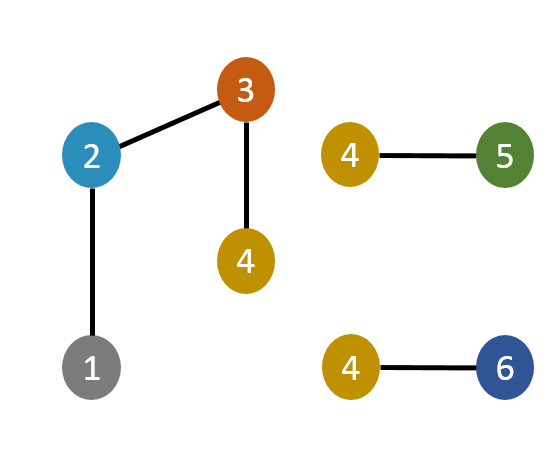}
		\caption{}
		\label{fig:v-c2}
	\end{subfigure}
	\caption{Vertex-cut partitioning Examples over a tree graph. (a) An articulation point with two edges, (b) Resulting vertex-cut partitioning from Fig.~\ref{fig:g1}, (c) An articulation point with three edges, and (d) Resulting vertex-cut partitioning from Fig.~\ref{fig:g2}.}
	\label{fig:v-c}
\end{figure*} 

For the purpose of our mathematical proof, we adopt the following assumption on the accuracy of the $ \mu $PMU measurements.
\begin{assumption}
	The $ \mu $PMU measurements are noiseless, i.e., the variance of the measurement noise associated with the $ \mu $PMU measured quantities is negligible.
\end{assumption}
This assumption is quite realistic as the signal to noise ratio in the $ \mu $PMU measurements is in the range of $ 40$ to $50 $ dB~\cite{Brown-2016, power_standards_lab_2018}. \ahmed{The errors in the instrumentation channel, e.g., CTs, PT, and control cables, are larger but stable which allows for filtering these errors utilizing consecutive measurements~\cite{farajollahi2018locating}.} 
Note that, this assumption is only used for our mathematical proof, but not assumed or invoked in any of the simulations. 

Estimating the voltages (states) of all the buses in the network is usually done by solving an optimization problem as in~\eqref{eq:psse}. Abstracting this concept, we can say that there is a mapping $ {\bf F}(\cdot) $ such that $ {\bf F}({\bf z}) = \hat{\bf v} $. Suppose $ \mathcal{P}\subset \mathcal{N} $ is the set that comprises the buses with $ \mu $PMUs installed. Since the $ \mu $PMU measurements are noiseless, the state estimation problem can be reduced to estimating the network voltages (states) only at the buses without $ \mu $PMUs ($ \mathcal{N}\backslash \mathcal{P} $). Such mapping is denoted by $ {\bf F}_r({\bf z}) = \hat{\bf v}_{\mathcal{N}\backslash \mathcal{P} } $ where $ {\bf v}_{\mathcal{N}\backslash \mathcal{P} } $ collects the voltages at all the nodes without $ \mu $PMUs. The next theorem shows the separability of this mapping over the partitioned graph.
\begin{theorem}
	\label{th:one}
	Suppose $ \{\mathcal{L}_k\}_{k=1}^{K} $ are the disjoint partitions that result from cutting the vertices in $ \mathcal{P} $. In addition, $ \mathcal{N}_k $ denotes the set of nodes connected to the edges in $ \mathcal{L}_k $, and $ \overline{\mathcal{N}}_k = \mathcal{N}_k \backslash \mathcal{P}$. Then, the mapping $ {\bf F}_r({\bf z}) = {\bf v}_{\mathcal{N}\backslash \mathcal{P}} $ is separable over the vertex-cut partitioning, i.e., for each set $ \overline{\mathcal{N}}_k $, the mapping $ {\bf F}_r(\cdot) $ can be written as $ {\bf F}_r^{(k)}({\bf z}^{(k)}) = \hat{{\bf v}}_{{\overline{\mathcal{N}}}_k} $, where $ {\bf z}^{(k)} $ comprises all the measurements taken at the buses $ \mathcal{N}_k $ and the edges $ \mathcal{L}_k $, and $ \hat{\bf v}_{\overline{\mathcal{N}}_k} $ collects the voltages at the buses in $ \overline{\mathcal{N}}_k $.
\end{theorem}
\begin{proof}
	Any measurement synthesizing function $ h_\ell({\bf v}) $ taken at a bus $ n $ is a function of the state at the bus $ n $ and the buses connected to bus $ n $. Similarly, for a current or power flow measurement taken at an edge $ (l, m) $, the synthesizing function is a function only of the state at the buses $ l $ and $ m $. 
	As the $ \mu $PMUs provide exact complex voltage measurements at the buses $ \mathcal{P} $, any measurement taken at a bus $ n \in \mathcal{N}_k$ on a line $ (l, m) \in \mathcal{L}_k $ depends only on the state at the buses $ \overline{\mathcal{N}}_k $. Therefore, the measurement synthesizing function of any measurement taken outside a certain partition $ \mathcal{L}_k $ does not involve the states (voltages) at the nodes $ \overline{\mathcal{N}}_k $. Hence, the state estimation mapping for $ \hat{v}_{\overline{\mathcal{N}}_k} $ is function only of the measurements $ {\bf z}^{(k)} $, which proves the theorem.
\end{proof}

The aforementioned theorem provides an insight regarding the separability of the state estimation problem in presence of $ \mu $PMUs. This separability is critical when a learning model is used to estimate the state of the network from the measurements. It is clear now that a learning model that estimates the voltage at a certain bus does not require knowledge of all the measurements in the network. Instead, by having an accurate measurement of the state (voltage) at a certain bus, all the measured quantities behind this bus can be discarded. This will play an important role in reducing the complexity of our learning model used for the task.

\section{Graph-Pruned Neural Networks for DSSE}
In this section, we present our novel learning model for DSSE. The graph-pruned neural network is composed of multiple layers whose connections reflect the distribution network connections. Let the input of the NN be denoted by $ {\bf x} $, and the NN produces an output $ {\bf y} $ using a stacked layered architecture in which each layer realizes a linear transformation and a point-wise nonlinearity. The vector $ {\bf y} $ is partitioned into $ N $ parts that represent features of each node in the graph, e.g., the voltage at the buses. The intermediate output at the $ t $-th layer of the NN is denoted by $ {\bf h}_t \in \mathbb{R}^{Nd_t}$ where $ d_t $ represents the dimension of each partition in $ {\bf h}_t $. Formally, the $ t $-th layer output is computed using the following transformation
\begin{equation}
{\bf h}_{t+1} = \sigma_l ( {\bf W}_t {\bf h}_t )
\end{equation}
where $ \sigma_t $ is a point-wise nonlinearity, and the matrix $ {\bf W}_t \in \mathbb{R}^{Nd_{t+1}\times Nd_t} $ is composed of $ N \times N $ blocks of size $ (d_{t+1} \times d_t) $. The $ (i, j) $ block in the matrix $ {\bf W}_t $ is zeroed out (pruned) if the nodes $ i $ and $ j $ are not connected in $ \mathcal{G} $, which justifies the name of the proposed learning model as \emph{graph-pruned} NN.

\begin{figure*}[t!]\label{fig:GPNN}
	\centering
	\begin{subfigure}{0.33\textwidth}
		\centering
		\includegraphics[height=2in]{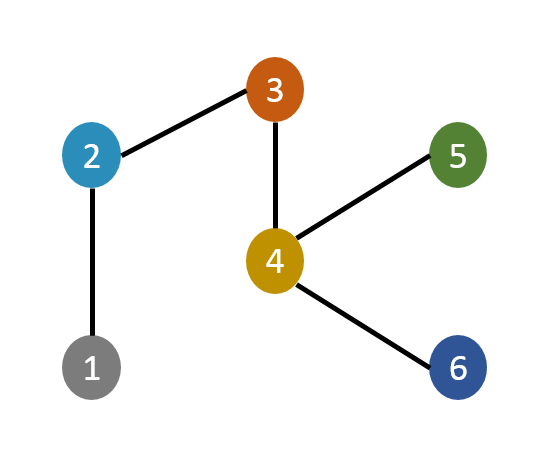}
		\caption{An example graph}
		\label{fig:example}
	\end{subfigure}~
	\begin{subfigure}{0.33\textwidth}
		\centering
		\includegraphics[height=2in]{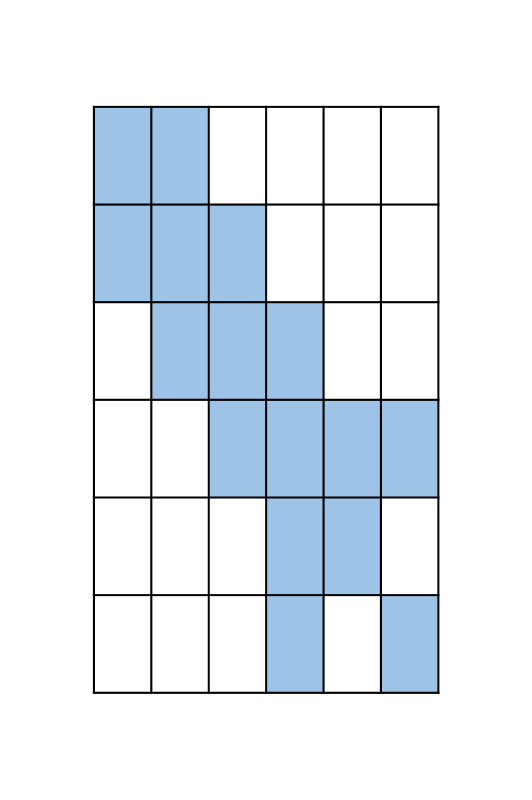}
		\caption{The sparsity of the weights}
		\label{fig:W}
	\end{subfigure}~
	\begin{subfigure}{0.33\textwidth}
		\centering
		\includegraphics[height=2in]{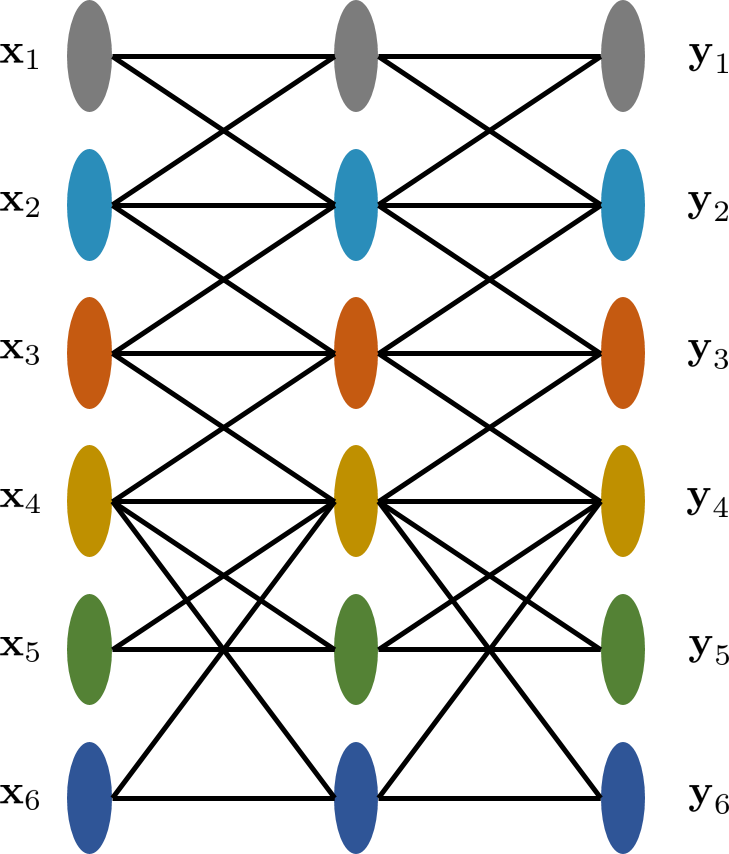}
		\caption{Graph-pruned NN for the graph in Fig.~\ref*{fig:example}}
		\label{fig:GNN}
	\end{subfigure}
	\caption{Example graph with the corresponding graph-pruned NN}
\end{figure*} 

\ahmed{\begin{example}
	Consider the graph in Fig.~\ref{fig:example}. Suppose a two-layer graph-pruned NN is designed to estimate some features of the nodes in the graph from signals (measurements) at the network nodes. The input vector $ {\bf x} := [ {\bf x}_1^T,\ {\bf x}_2^T,\ \ldots,\ {\bf x}_6^T ]^T$, and the output vector $ {\bf y} $ is also composed of six components $ {\bf y}_i $ for $ 1 \leq i \leq 6 $. The output of the NN can be written as
	$$ {\bf y} = \sigma_2 \big( {\bf W}_2 \ \sigma_1 ( {\bf W}_1 {\bf x} + {\bf b}_1) + {\bf b}_2\big)
	$$
	where the structure of the weight matrices $ {\bf W}_i $ is depicted in Fig.~\ref{fig:W}. In addition, let $ {\bf h}_{t,i} $ denote the $ i $-th block of the output of the $ t $-th layer in the NN. Fig.~\ref{fig:GNN} shows the dependency of the block of each layer in the NN on the block of the previous layer. For example, $ {\bf h}_{1, 2} $ is function of $ {\bf x}_1$ and ${\bf x}_2 $ only. Notice that since the network is composed of two layers, each output block $ {\bf y}_i $ is function of the inputs related to nodes that are at most two hops away in the graph. For example, $ {\bf y}_4 $ is function of all inputs except $ {\bf x}_1 $, meaning that any change in $ {\bf x}_1 $ will not affect the output ${\bf y}_4 $ as long as only two layers are used. 
\end{example}
}

The graph convolutional NN (GCNN)~\cite{kipf2016semi, ioannidis2018recurrent} learning approaches are designed to process data defined over graphs. These models also lead to sparsification of the weight matrices connecting the hidden layers in the neural network. In addition to being derived from the physical model governing the operations, the proposed learning model leads to a more general parameterization.
For instance, the GCNN model proposed in~\cite{kipf2016semi} is tantamount to enforcing the blocks of the weight matrices to be scaled versions of a single matrix. For instance, if the small gray blocks in Fig.~\ref{fig:W} are constrained to be scaled versions of each other, then the proposed method becomes equivalent to the GCNN in~\cite{kipf2016semi}. On the other hand, if the vertically aligned blocks in the weight matrices of PAWNN are chosen to be scaled versions of a fixed matrix, then the learning model of~\cite{ioannidis2018recurrent} emerges as a special case. For example, for the PAWNN in Fig.~2, if we constrain the blocks $ (3, 4) $, $ (5, 4) $, and $ (6, 4) $ of the weights matrix in Fig.~\ref{fig:W} to be scaled versions of the block $ (4, 4) $, and similarly for all other vertically aligned blocks, then the model reduces to the graph NN proposed in~\cite{ioannidis2018recurrent}.

\subsection{Required number of layers}
As established in Section~\ref{sec:PDSSE}, the DSSE problem can be partitioned by installing $ \mu $PMUs that essentially break the dependencies of the estimated state at a certain bus on any measurement taken outside its partition. Therefore, the graph-pruned NN has the potential to realize mappings such that the estimated state at a certain bus is a function only of the measurements taken in the same partition. In order to characterize the number of layers required to realize the DSSE mapping for a network with $ \mu $PMUs installed, we introduce the following definitions.

\begin{definition}
	The \emph{eccentricity} of a vertex $ v $ in $ \mathcal{G}({\mathcal{N},\mathcal{L}}) $ is the maximum shortest path length from $ v $ to all vertices in $ \mathcal{N} $.
\end{definition}
\vspace{-10pt}
\begin{definition}
	The \emph{diameter} of a graph $ \mathcal{G}({\mathcal{N},\mathcal{L}}) $ is the maximum eccentricity of all the vertices in $ \mathcal{N} $.
\end{definition}
\vspace{-10pt}
\begin{definition}
	The \emph{diameter} of a vertex-cut partitioning is the maximum diameter of the subgraphs $ \mathcal{G}_k(\mathcal{N}_k, \mathcal{L}_k) $ for $ 1 \leq k \leq K $, which is introduced by cutting the vertices in a set $ \mathcal{P} \subseteq \mathcal{N} $, and it is denoted by $ \text{dia}(\mathcal{P}) $.
\end{definition}

For example, the diameter of the vertex-cut partitioning in Fig.~\ref{fig:v-c1} is $ 2 $, while the diameter of the vertex-cut partitioning in Fig.~\ref{fig:v-c2} is $ 3 $. Therefore, a two-layer graph-pruned NN, as in Fig.~\ref{fig:GNN}, can penitentially approximate the mapping between the measurements and the states if a $ \mu $PMU unit is installed at the vertex that was cut in Fig.~\ref{fig:v-c1}. However, with the cut in Fig.~\ref{fig:v-c2}, at least three-layer graph-pruned NN is needed. Now, it is reasonable to ask how to place available $ \mu $PMUs such that the diameter of the resulting vertex-cut partitioning is minimized. In the next subsection, we present a greedy algorithm that provides a simple approximate solution for this problem.

\subsection{Greedy algorithm for $ \mu $PMU placement}
As shown in Theorem~\ref{th:one}, the placement of $ \mu $PMUs in the distribution network makes the DSSE problem separable. In the experiments section, we will show that under realistic setup the decoupled DSSE subproblems are almost equivalent to the original formulation. Therefore, the $ \mu $PMUs need to be placed in the feeder such that the resulting subproblems are balanced. In other words, we tackle the problem of minimizing the diameter of the resulting vertex-cut partitioning given a certain budget of $ \mu $PMUs. We present a greedy algorithm that provides an approximate solution of this placement problem.

The proposed greedy algorithm tackles the problem of installing $ \mu $PMUs one at a time. It is clear that the optimal placement of a $ \mu $PMU in order to reduce the diameter of the resulting vertex-cut partitioning is to install it in the middle of the longest shortest path in all the partitions. Therefore, our algorithm starts by finding the maximum length shortest path in the network, and the $ \mu $PMU is installed in the middle of this path. Then, the process continues by finding the maximum length shortest path in all the resulting subgraphs, and then placing the next $ \mu $PMU along the maximum length path in all the subgraphs. The process continues until the available budget of $ \mu $PMUs is exhausted. Algorithm~\ref{Alg:greedy} summarizes the main steps of the proposed approach.

\SetKw{input}{Input: }
\SetKw{output}{Output: }
\SetKw{init}{Initialization: }
\DontPrintSemicolon
\begin{algorithm}[h]
	\caption{Greedy Algorithm for $ \mu $PMU Placement}
	{\footnotesize
		\input{\rm graph $ \mathcal{G}(\mathcal{N}, \mathcal{L}) $ and $ K \geq 1 $ budget of $ \mu $PMUs}\;
		\output{\rm set $ \mathcal{S}\subseteq\mathcal{N} $ where $ |\mathcal{S}| = K $}\;
		\init{\rm $ \mathcal{S} = \phi$}\;
		\Repeat{$|\mathcal{S}| = K $}{
			[S1]	Determine the maximum length shortest path in all the subgraphs $ \mathcal{G}_k(\mathcal{N}_k, \mathcal{L}_k) $ resulting from cutting the vertices in $ \mathcal{S} $\;
				\vspace{.2cm}		
			[S2]	Place a $ \mu $PMU in the middle of the longest path identified in [S1] \;
			\vspace{.2cm}	
		}
	}	
	\label{Alg:greedy}
\end{algorithm}

In order to find maximum length shortest path in a tree a simple algorithm is used. Let us consider a subgraph $ \mathcal{G}_k(\mathcal{N}_k, \mathcal{L}_k) $. First, we choose a random starting point $ n\in \mathcal{N}_k $ and perform depth-first search (DFS) to find the eccentricity of $ n $ which is achieved for the path from $ n $ to vertex $ n' \in \mathcal{N}_k$. Now, we use DFS to find the eccentricity of node $ n' $. The length of the path achieving maximum shortest path length from $ n' $ is the diameter of $ \mathcal{G}_k $. Therefore, the complexity of each step is $ \mathcal{O}(|\mathcal{N}|) $, and the process is repeated $ K $ times. Hence, the total complexity of Algorithm~\ref{Alg:greedy} is $ \mathcal{O}(K|\mathcal{N}|) $.

\section{Experimental Results}

	\begin{figure}[t]
		\centering
		\includegraphics[width=0.85\columnwidth]{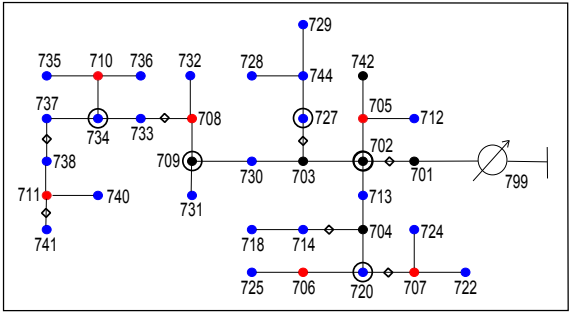}
		\caption{IEEE-37 distribution feeder. Nodes in blue are with loads, and red nodes represent buses with DER installed.}
		\label{fig:Network}
		\vspace{-13pt}
	\end{figure}
In this section, the proposed graph-pruned NN is utilized to estimate the state of the benchmark IEEE-37 distribution feeder. {This network was recommended by the Test Feeder Working Group of the Distribution System Analysis Subcommittee of the IEEE PES for evaluating the performance of the state estimation algorithms~\cite{Feeder-WG}.} The feeder has several delta-connected loads and is known to be highly unbalanced. The load buses are blue-colored in Fig.~\ref{fig:Network}. In addition, some nodes in the feeder feature different types of connections, i.e., single-, two-, and three-phase connections. Renewable energy sources (RES) are installed at six different buses, which are colored in red in Fig.~\ref{fig:Network}. \ahmed{In addition, the buses with delta-connected loads are colored in blue in Fig.~\ref{fig:Network}.}


We evaluate the performance of the proposed greedy algorithm for placing $ \mu $PMUs in the distribution feeder. We compare the diameter of the partitioning induced by our proposed approach against the one resulting from optimal placement of $ \mu $PMUs. For the optimal placement, exhaustive search is used in order to identify the placement that minimizes the diameter of the resulting vertex-cut partitioning. \ahmed{As the objective of this optimal placement is to minimize the diameter of the resulting partitioning which matches the proposed neural network architecture, there has been no other placement algorithm in the literature for this purpose.}
Table~\ref{tab:placement} compares the placement of $ \mu $PMUs using our proposed approach against the optimal placement in terms of diameter of the resulting vertex-cut partitioning \ahmed{for different budgets in terms of the total number of $ \mu $PMUs.} 
The set $ \mathcal{S}_g $ denote the buses with $ \mu $PMUs installed using the greedy algorithm, while $ \mathcal{S}_o $ denote the optimal placement of $ \mu $PMUs. 
The results shows that the proposed approach produces optimal placement in all cases except for $ 2 $ and $ 6 $ $ \mu $PMUs where the diameter is larger by only one.  \ahmed{Moreover, for $ 7 $, $ 8 $, and $ 9 $ $ \mu $PMUs, both approaches produce placement patterns that yield partitions of diameter $ 3 $.}

\begin{table}[htbp]
	\renewcommand{\arraystretch}{1.2}
	\centering
	\caption{Optimal and greedy $ \mu $PMU placement.}
	\begin{tabular}{c c c c c }
		\toprule
		{\textbf{\# $ \mu $PMUs}} & {\textbf{$ \mathcal{S}_g $}} &{$\textit{dia}( \mathcal{S}_g $)} & {\textbf{$ \mathcal{S}_o $}} &{$\textit{dia}( \mathcal{S}_o $)}\\
		\midrule
		1     & \{709\}  & $ 8 $  &  \{709\} &  $ 8 $ \\
		2     & \{702, 709\}  & $ 7  $      &  \{702, 708\} &  $ 6 $\\
		3     & \{702, 709, 734\}  & $ 5 $  &  \{702, 709, 734\}&  $ 5 $ \\
		\multirow{2}{*}{4}     & \{702, 709, 720  & \multirow{2}{*}{$ 5 $}       &  \{702, 709, 720, &  \multirow{2}{*}{$ 5 $}  \\
		& 734\}&  & 734\}& \\
		\multirow{2}{*}{5}      & \{702, 709, 720, & \multirow{2}{*}{$ 4 $}       &  \{702, 709, 720, &  \multirow{2}{*}{$ 4 $}  \\
		& 727, 734\}& &  727, 734\} & \\
		\multirow{2}{*}{6}      & \{702, 709, 720, & \multirow{2}{*}{$ 4 $}       &  \{703, 708, 713, &  \multirow{2}{*}{$ 3 $} \\
		&  727, 734, 738\}& &  720, 733, 737\}& \\
		\bottomrule
	\end{tabular}%
	\vspace{-5pt}
	\label{tab:placement}%
\end{table}%

Training samples were generated using the load and renewable generation dataset available in~\cite{bank2013development} modulated by the nominal values of the loads. \ahmed{The training and testing instants include diverse loading situations in different times of the day}.The power flow solver~\cite{Garces2015} was used to find the voltage profile (network states), and then the noisy measurements were generated using the measurement synthesizing functions~\eqref{eq:meas1} and~\eqref{eq:meas2}\footnotemark. The variance of the noise added to the $ \mu $PMU measurements was set to be $ 10^{-6} $, and the variances of the current magnitude and pseudo-measurements noise were $ 10^{-3} $ and $ 10^{-2} $, respectively. In each scenario, a total of $ 100,000 $ samples were used to train graph-pruned NNs using the TensorFlow~\cite{tensorflow2015-whitepaper} software library with $ 90 \%$ of the data used for training and $ 10 \% $ for validation. The ADAM optimizer~\cite{kingma2014adam} is used to train the neural networks used in this section. 
The estimate of the voltages at a bus is represented using $ 6 $ outputs representing the real and imaginary parts of the voltage phasor at each phase. The graph structure of the network is imposed on the NN connectivity. \ahmed{Hence, the number of neurons at each layer of the neural network is proportional to the number of buses, i.e., the width of the $ t $-th hidden layer is $ d_t N $ where each bus is represented by $ d_t $ neurons. Then, the neurons that represent bus $ n $ at the $ t $-th layer are connected only to the neurons representing the buses neighboring $ n $ in the graph $ \mathcal{G} $.}

\footnotetext{\ahmed{Although the linear solver is not extremely accurate, the obtained solutions represent realistic network states which are later used to generate measurements.}}

\ahmed{In order to assess the performance of the proposed approach, we use two baselines. On is the Wirtinger-Flow Gauss-Newton algorithm proposed in~\cite{dzafic2018hybrid}. The other is the data-driven optimization approach in~\cite{Zamzam-2018} where a shallow neural network is trained to initialize the Gauss-Newton approach. The width of the hidden layer in that approach is $ 2048 $, and the training parameter $ \epsilon $ is chosen to be $ \frac{1}{2} $ which represents the relaxation of the training cost function, cf.~\cite{Zamzam-2018}.} We define the average estimation accuracy $ \boldsymbol{\nu} $ of each algorithm as follows.
\begin{equation}
\boldsymbol{\nu} = \frac{1}{N} \sum_{i=1}^{N} \| \hat{\bf v}_i- {\bf v}_i^{\text{true}}\|_2^2
\end{equation}
where $ \hat{\bf v}_i $ is the estimated voltage profile from the noisy measurements generated using $ {\bf v}_{i}^{\text{true}} $. 

We assess the performance of the proposed graph-pruned NN in two different scenarios of measurements. In the first case (Scenario A), we employ $ 5 $ $ \mu $PMUs installed according to the proposed greedy algorithm, which represents optimal placement. In Fig.~\ref{fig:Network}, the buses where the $ \mu $PMUs are installed in Scenrario A are circled, and the lines where the current magnitudes are measured have rhombuses on them. The net load and renewable energy generation at all the phases of the buses with loads or RESs installed are used as pseudo-measurements. The number of layers in the graph-pruned NN in this scenario is $ 4 $ as the diameter of the resulting partitioning of this placement is equal to $ 4 $, which means that $ 4 $ layers of the graph-pruned NN are enough to represent the mapping between the measurements and the states. In this scenario, the total number of measurements is $ 103 $, which consist of $ 3 $ complex measurements of voltages at $ 5 $ buses with $ \mu $PMUs, $ 3 $ real measurement of current magnitudes installed at $ 7 $ locations in the network, and $ 26 $ complex pseudo-measurements at the buses with loads or RESs installed and without $ \mu $PMUs. \ahmed{The number of neurons representing each bus at the hidden layers is $ 48 $, $ 24 $, $ 12 $, and $ 6 $, respectively.}

Table~\ref{tab:results} shows the average performance of the proposed physic-aware learning approach, the hybrid data-driven and optimization method~\cite{Zamzam-2018} (SNN + G-N), and the Gauss-Newton (G-N) algorithm over $ 1000 $ cases of Scenario A. \ahmed{Simple feed-forward neural networks approaches require a lot of training data and computational resources. In addition, they often suffer from exploding or diminishing gradients. This results in bad estimates for the state of the network. For example, the average estimation accuracy of the state using a $ 4 $-layer feed-forward NN, with number of neurons similar to the $ 4 $-layer GPNN, was $ 2.69\times10^{-1} $ for noiseless measurements in Scenario A. Hence, we did not include a feed-forward NN in our comparisons.} The Gauss-Newton algorithm is initialized using the flat voltage profile. Clearly, the proposed learning method achieves superior performance where the accuracy of estimation is an order of magnitude better than the state-of-the-art Gauss-Newton approach. In addition, since the proposed learning method alleviates almost all the computational burden at the estimation time by shifting it to the training time, the running time of the proposed approach is three orders of magnitude less than the optimization-based approach.

\begin{table}[t]
	\renewcommand{\arraystretch}{1.25}
	\newcolumntype{C}[1]{>{\centering\arraybackslash}p{#1}}
	\caption{Performance comparison of different state estimators evaluated in Scenario A}
	\begin{center}
		\begin{tabular}{p{1.75cm}  C{2.25cm} C{2.25cm}}
			\toprule
			{\textbf{Method}}&  {{$\boldsymbol{\nu}$}} & {Time (ms)}\\
			\midrule 
			{\bf PAWNN}     &  $ 1.273\times 10^{-3} $ & $ 1.146  $  \\
			{\bf SNN + G-N}     &  $ 8.341\times 10^{-2} $ & $ 124
			$     \\
			{\bf G-N}     &  $ 5.833\times 10^{-1} $ & $ 866  $     \\
			\bottomrule
		\end{tabular}%
	\end{center}
	\label{tab:results}%
	\vspace{-20pt}
\end{table}%

\ahmed{In the second case (Scenario B), we evaluate the performance of the proposed graph-pruned neural network under a suboptimal $ \mu $PMU placement pattern, where $ \mu $PMUs are at the buses \{701, 704, 708, 738, 744\} and the other measurements are left unaltered. The diameter of the partitioning resulting from this $ \mu $PMU placement pattern is $ 6 $. This means that at least $ 6 $ layers of graph-pruned NN are required to realize the mapping between the measurements and the estimated state. The number of neurons that represent each bus in the first and second hidden layers of the $ 2 $-layer GPNN is $ 48 $ and $ 12 $, respectively, while for the $ 4 $-layer graph-pruned NN we use $ 48 $, $ 24 $, $ 12 $, and $ 6 $ neurons to represent each bus at the layers from the first to the fourth. For the $ 6 $-layer GPNN the number of neurons used to represent each bus at the layers from the first to the sixth are $ 48 $, $ 24 $, $ 12 $, $ 6 $, $ 6 $, and $ 6 $, respectively.}

\ahmed{Table ~\ref{tab:resultsB} summarizes the quality of state estimation using $ 2 $-, $ 4 $-, and $ 6 $-layer graph-pruned NNs for Scenario B. It can be seen that estimation of quality closer to the accuracy achieved in Scenario A is only possible using $ 6 $-layer GPNN due to the diameter ($ 6 $) of partitioning resulting from the considered placement. The small difference in the accuracy relative to Scenario A is likely due to the suboptimal placement of the $ \mu $PMUs which is not balanced over the network.}
		
\begin{table}[t]
	\renewcommand{\arraystretch}{1.25}
	\newcolumntype{C}[1]{>{\centering\arraybackslash}p{#1}}
	\caption{Performance comparison of different state estimators evaluated in Scenario B. PAWNN (4L) means with 4 Layers.} 
	\begin{center}
		\begin{tabular}{p{1.75cm}  C{2.25cm} C{2.25cm}}
			\toprule
			{\textbf{Method}}&  {{$\boldsymbol{\nu}$}} & {Time (ms)}\\
			\midrule 
			{\bf PAWNN} (2L)     &  $ 2.411\times 10^{-1} $ & $ 0.731  $  \\
			{\bf PAWNN} (4L)     &  $ 5.170\times 10^{-2} $ & $ 1.241  $  \\
			{\bf PAWNN} (6L)     &  $ 5.330\times 10^{-3} $ & $ 1.259  $  \\
			{\bf SNN + G-N}     &  $ 2.860\times 10^{-2} $ & $ 589	$     \\
			{\bf G-N}     &  $ 4.489\times 10^{-1} $ & $ 2891  $     \\
			\bottomrule
		\end{tabular}%
	\end{center}
	\vspace{-20pt}
	\label{tab:resultsB}%
\end{table}%

In order to show the quality of estimates provided by the proposed Graph-pruned neural network, we present the estimate of the voltage magnitudes and angles at phase (c) at all buses. Fig.~\ref{fig:estimate} depicts the estimated voltage magnitudes and angles using the Gauss-Newton method and the proposed Graph-pruned NN approach. Also, the absolute estimation error of the magnitudes and angles are shown in faded colors. The results show superior estimation performance for the proposed approach.
\begin{figure}[htbp]\vspace{-10pt}
	\centering
	\begin{subfigure}{0.5\textwidth}
		\centering
		\includegraphics[height=1.75in]{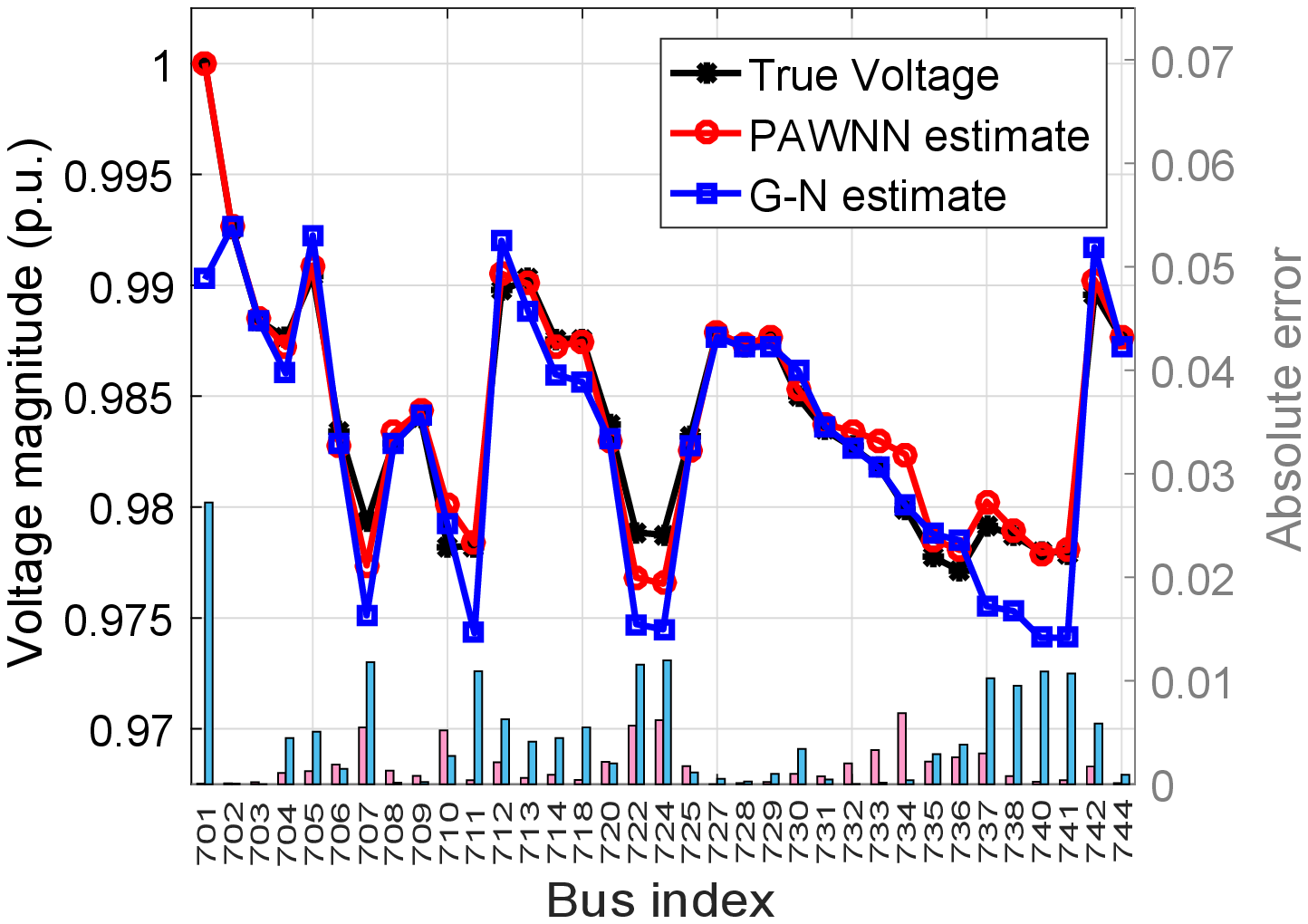}
		\caption{Voltage magnitudes}
		\label{fig:mag}
	\end{subfigure}
	\begin{subfigure}{0.5\textwidth}
		\centering
		\includegraphics[height=1.75in]{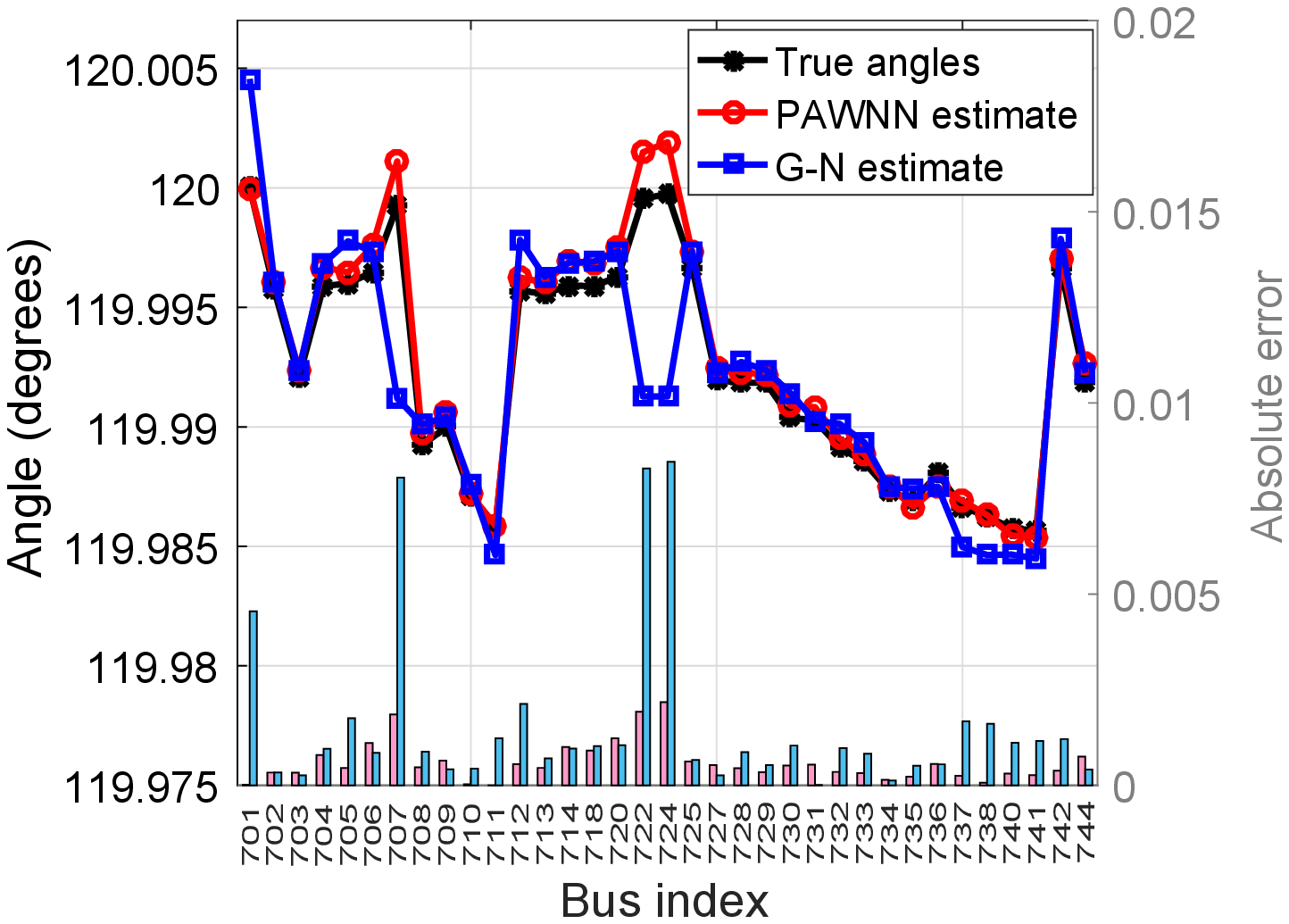}
		\caption{Voltage angles}
		\label{fig:phase}
	\end{subfigure}
	\caption{Estimation of the voltage magnitudes and angles at phase (b) of all buses in the IEEE-37 feeder. \small
		{(The absolute estimation errors are depicted in pink and light blue for the proposed approach and the Gauss-Newton solver, respectively.)}}
	\label{fig:estimate}
	\vspace{-20pt}
\end{figure}
	
{\subsection{Robustness of PAWNN}
As discussed earlier, the proposed approach is inherently robust against measurements failure or attacks. That is, erroneous measurements are not propagated in the neural network more than the number of layers. Therefore, only the estimation of voltages in the neighborhood is affected. In order to showcase the robustness of the approach, we tested the proposed learning model in a scenario where the measurements of the $ \mu $PMU installed at bus $ 734 $ are corrupted by Gaussian noise with a standard deviation of $ 10 $. \ahmed{This $ \mu $PMU provides $ 3 $ complex measurements representing the voltage phasors at the three phases.} While the proposed approach is oblivious to the noise level of each measurement, the weights in the weighted least squares formulation used by the Gauss-Newton approach were adjusted to account for the huge noise variance of the measurements at bus $ 734 $. The estimation of the voltage magnitude along all buses at phase (b) for both approaches is depicted in Fig.~\ref{fig:estimateR}. It is noticeable that the estimation of voltages at the buses surrounding the bus $ 734 $ are the only affected buses, while the G-N estimate is totally corrupted by corrupting the measurements of only one measuring unit despite adjusting the weights in the WLS formulation to be cognizant of these corrupt measurements. .
\begin{figure}[htbp]
		\centering
		\includegraphics[height=1.75in]{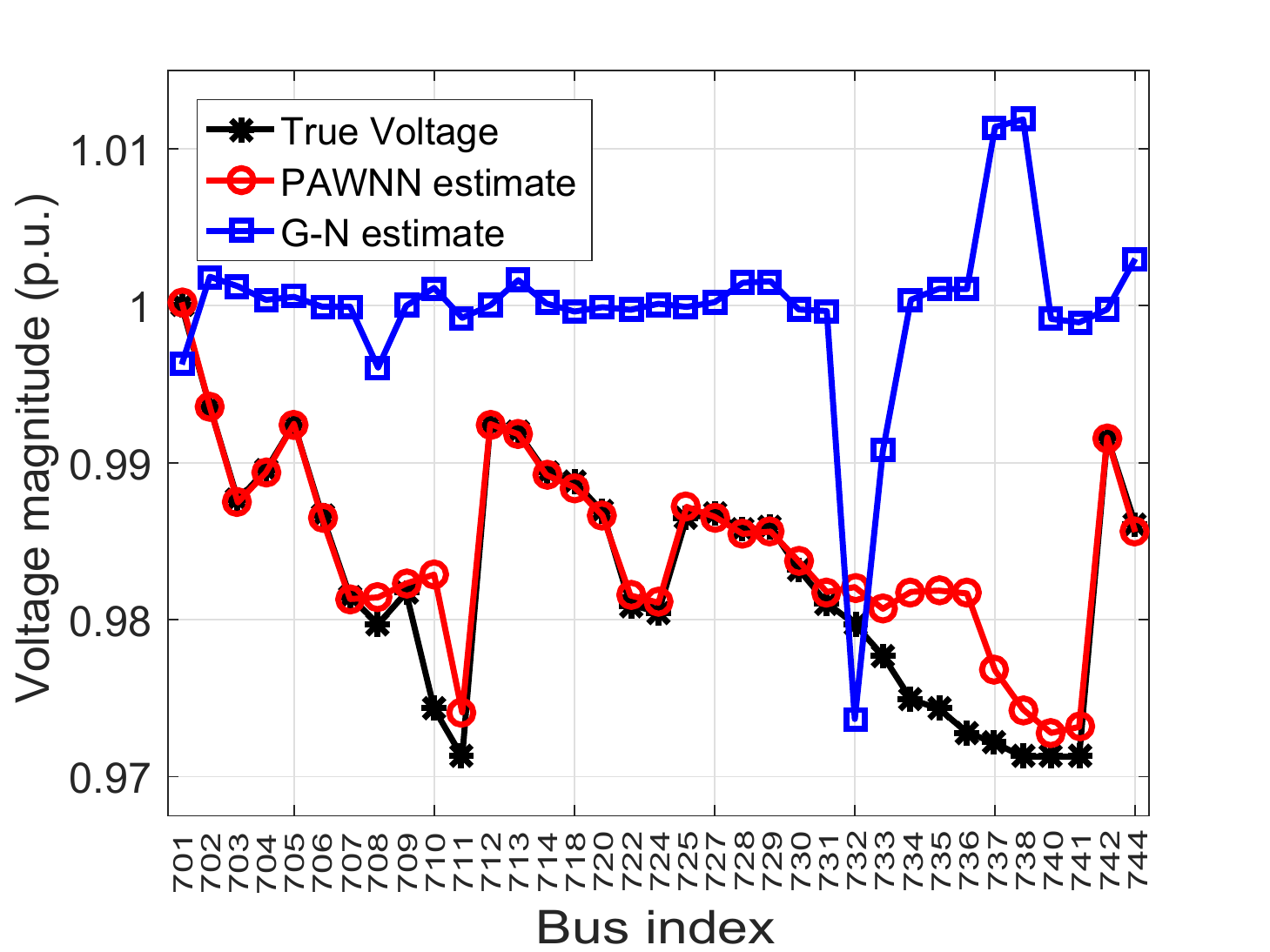}
	\caption{Estimation of the voltage magnitudes at phase (b) of all buses under corrupted $ \mu $PMU measurements at bus 734.}
	\label{fig:estimateR}
	\vspace{-15pt}
\end{figure}
 
\section{Conclusions}
	
This paper proposed a novel learning model that facilitates real-time monitoring of distribution network operation. The graph-pruned NN approach utilizes the approximate separability of the DSSE problem resulting from installing $ \mu $PMUs at some buses. By pruning the unneeded NN connections, the resulting model prevents over-fitting behavior exploiting the available knowledge regarding the network physics. A greedy algorithm was proposed for installing $ \mu $PMUs in order to minimize the diameter of the resulting partitioning of the distribution feeder. Simulation results corroborate the efficacy of the greedy algorithm for finding near-optimal placement solutions. Also, the proposed PAWNN approach shows superior performance in estimating the network state from few noisy real-time measurements and pseudo-measurements on the IEEE-37 distribution feeder. In addition, the proposed approach was shown to be robust against corrupted measurements.
%
%
%
%

\vspace{-5pt}
\bibliographystyle{IEEEtran}
	\bibliography{GraphNN}

\begin{thebibliography}{10}
\providecommand{\url}[1]{#1}
\csname url@samestyle\endcsname
\providecommand{\newblock}{\relax}
\providecommand{\bibinfo}[2]{#2}
\providecommand{\BIBentrySTDinterwordspacing}{\spaceskip=0pt\relax}
\providecommand{\BIBentryALTinterwordstretchfactor}{4}
\providecommand{\BIBentryALTinterwordspacing}{\spaceskip=\fontdimen2\font plus
\BIBentryALTinterwordstretchfactor\fontdimen3\font minus
  \fontdimen4\font\relax}
\providecommand{\BIBforeignlanguage}[2]{{%
\expandafter\ifx\csname l@#1\endcsname\relax
\typeout{** WARNING: IEEEtran.bst: No hyphenation pattern has been}%
\typeout{** loaded for the language `#1'. Using the pattern for}%
\typeout{** the default language instead.}%
\else
\language=\csname l@#1\endcsname
\fi
#2}}
\providecommand{\BIBdecl}{\relax}
\BIBdecl

\bibitem{baran1994}
M.~E. Baran and A.~W. Kelley, ``State estimation for real-time monitoring of
  distribution systems,'' \emph{IEEE Trans. on Power Systems}, vol.~9, no.~3,
  pp. 1601--1609, Aug 1994.

\bibitem{li1996state}
K.~Li, ``State estimation for power distribution system and measurement
  impacts,'' \emph{IEEE Transactions on Power Systems}, vol.~11, no.~2, pp.
  911--916, May 1996.

\bibitem{singh2009choice}
R.~Singh, B.~Pal, and R.~Jabr, ``Choice of estimator for distribution system
  state estimation,'' \emph{IET Generation, Transmission \& Distribution},
  vol.~3, no.~7, pp. 666--678, July 2009.

\bibitem{kekatos2013distributed}
V.~Kekatos and G.~B. Giannakis, ``Distributed robust power system state
  estimation,'' \emph{IEEE Trans. on Power Systems}, vol.~28, no.~2, pp.
  1617--1626, May 2013.

\bibitem{Wang2018}
G.~Wang, A.~S. Zamzam, G.~B. Giannakis, and N.~D. Sidiropoulos, ``Power system
  state estimation via feasible point pursuit: Algorithms and {Cram\'er-Rao}
  bound,'' \emph{IEEE Transactions on Signal Processing}, vol.~66, no.~6, pp.
  1649--1658, Mar 2018.

\bibitem{dzafic2018hybrid}
I.~Dzafic, R.~A. Jabr, and T.~Hrnjic, ``Hybrid state estimation in complex
  variables,'' \emph{IEEE Transactions on Power Systems}, 2018,
  {DOI}:10.1109/TPWRS.2018.2794401.

\bibitem{baran1995branch}
M.~E. Baran and A.~W. Kelley, ``A branch-current-based state estimation method
  for distribution systems,'' \emph{IEEE Trans. on Power Systems}, vol.~10,
  no.~1, pp. 483--491, Feb 1995.

\bibitem{wang2004revised}
H.~Wang and N.~N. Schulz, ``A revised branch current-based distribution system
  state estimation algorithm and meter placement impact,'' \emph{IEEE Trans. on
  Power Systems}, vol.~19, no.~1, pp. 207--213, Feb 2004.

\bibitem{manitsas2012distribution}
E.~Manitsas, R.~Singh, B.~C. Pal, and G.~Strbac, ``Distribution system state
  estimation using an artificial neural network approach for pseudo measurement
  modeling,'' \emph{IEEE Trans. on Power Systems}, vol.~27, no.~4, pp.
  1888--1896, Nov 2012.

\bibitem{Barbeiro2014}
P.~N.~P. Barbeiro, J.~Krstulovic, H.~Teixeira, J.~Pereira, F.~J. Soares, and
  J.~P. Iria, ``State estimation in distribution smart grids using
  autoencoders,'' in \emph{2014 IEEE 8th International Power Engineering and
  Optimization Conference (PEOCO2014)}, March 2014, pp. 358--363.

\bibitem{Zamzam-2018}
A.~S. Zamzam, X.~Fu, and N.~D. Sidiropoulos, ``Data-driven learning-based
  optimization for distribution system state estimation,'' \emph{IEEE Trans. on
  Power Systems}, Early Access.

\bibitem{zhang2018real}
L.~Zhang, G.~Wang, and G.~B. Giannakis, ``Real-time power system state
  estimation and forecasting via deep neural networks,'' \emph{arXiv preprint
  arXiv:1811.06146}, 2018.

\bibitem{wang2017robust}
G.~Wang, G.~B. Giannakis, and J.~Chen, ``Robust and scalable power system state
  estimation via composite optimization,'' \emph{arXiv preprint
  arXiv:1708.06013}, 2017.

\bibitem{gao2018large}
H.~Gao, Z.~Wang, and S.~Ji, ``Large-scale learnable graph convolutional
  networks,'' in \emph{Proceedings of the 24th ACM SIGKDD International
  Conference on Knowledge Discovery \& Data Mining}.\hskip 1em plus 0.5em minus
  0.4em\relax ACM, 2018, pp. 1416--1424.

\bibitem{DallAnese13}
E.~Dall'Anese, H.~Zhu, and G.~B. Giannakis, ``Distributed optimal power flow
  for smart microgrids,'' \emph{IEEE Trans. on Smart Grid}, vol.~4, no.~3, pp.
  1464--1475, Sep 2013.

\bibitem{Zamzam-2016}
A.~S. Zamzam, N.~D. Sidiropoulos, and E.~Dall'Anese, ``Beyond relaxation and
  newton-raphson: Solving {AC OPF} for {M}ulti-phase {S}ystems with
  {R}enewables,'' \emph{IEEE Trans. on Smart Grid}, vol.~9, no.~5, pp. 3966 --
  3975, Sept. 2018.

\bibitem{Chartrand-1985}
G.~Chartrand, \emph{Introductory graph theory}.\hskip 1em plus 0.5em minus
  0.4em\relax New York: Dover, 1985.

\bibitem{Brown-2016}
M.~Brown, M.~Biswal, S.~Brahma, S.~J. Ranade, and H.~Cao, ``Characterizing and
  quantifying noise in {PMU} data,'' in \emph{2016 IEEE Power and Energy
  Society General Meeting (PESGM)}, July 2016, pp. 1--5.

\bibitem{power_standards_lab_2018}
\BIBentryALTinterwordspacing
``micropmu data sheet,'' Nov 2018. [Online]. Available:
  \url{https://www.powerstandards.com/download-center/micropmu/}
\BIBentrySTDinterwordspacing

\bibitem{farajollahi2018locating}
M.~Farajollahi, A.~Shahsavari, E.~M. Stewart, and H.~Mohsenian-Rad, ``Locating
  the source of events in power distribution systems using micro-pmu data,''
  \emph{IEEE Transactions on Power Systems}, vol.~33, no.~6, pp. 6343--6354,
  2018.

\bibitem{kipf2016semi}
T.~N. Kipf and M.~Welling, ``Semi-supervised classification with graph
  convolutional networks,'' \emph{arXiv preprint arXiv:1609.02907}, 2016.

\bibitem{ioannidis2018recurrent}
V.~N. Ioannidis, A.~G. Marques, and G.~B. Giannakis, ``A recurrent graph neural
  network for multi-relational data,'' \emph{arXiv preprint arXiv:1811.02061},
  2018.

\bibitem{Feeder-WG}
K.~P. Schneider \emph{et~al.}, ``Analytic considerations and design basis for
  the {IEEE} distribution test feeders,'' \emph{IEEE Transactions on Power
  Systems}, vol.~33, no.~3, pp. 3181--3188, May 2018.

\bibitem{bank2013development}
J.~Bank and J.~Hambrick, ``Development of a high resolution, real time,
  distribution-level metering system and associated visualization, modeling,
  and data analysis functions,'' National Renewable Energy Laboratory (NREL),
  Golden, CO., Tech. Rep., 2013.

\bibitem{Garces2015}
A.~Garces, ``A linear three-phase load flow for power distribution systems,''
  \emph{IEEE Trans. on Power Systems}, vol.~31, no.~1, pp. 827--828, Jan 2016.

\bibitem{tensorflow2015-whitepaper}
\BIBentryALTinterwordspacing
M.~Abadi \emph{et~al.}, ``{TensorFlow}: Large-scale machine learning on
  heterogeneous systems,'' 2015, software available from tensorflow.org.
  [Online]. Available: \url{https://www.tensorflow.org/}
\BIBentrySTDinterwordspacing

\bibitem{kingma2014adam}
D.~P. Kingma and J.~Ba, ``Adam: A method for stochastic optimization,''
  \emph{arXiv preprint arXiv:1412.6980}, 2014.

\end{thebibliography}
\end{document}